\documentclass[12pt]{amsart}
\usepackage{amssymb}

\newcommand{\R}{\mathbb{{R}}}

\newtheorem{theorem}{Theorem}
\newtheorem{lemma}{Lemma}

\theoremstyle{definition}

\theoremstyle{remark}
\newtheorem{remark}[theorem]{Remark}

\begin{document}
\title{Fourier transforms on an amalgam type space}
\author{E. Liflyand}

\address{Department of Mathematics,
Bar-Ilan University, 52900 Ramat-Gan, Israel}
\email{liflyand@math.biu.ac.il}



\begin{abstract}
We introduce an amalgam type space, a subspace of $L^1(\mathbb
R_+).$ Integrability results for the Fourier transform of a function
with the derivative from such an amalgam space are proved. As an
application we obtain estimates for the integrability of
trigonometric series.
\end{abstract}

\maketitle

\section{Introduction}

Let us start with a brief overview of the (100 years) old problem of
integrability of trigonometric series. Given a trigonometric series

\begin{eqnarray}\label{fouser} a_0/2+\sum_{n=1}^\infty(a_n\cos nx+b_n\sin nx),
\end{eqnarray}
find assumptions on the sequences of coefficients
$\{a_n\},\{b_n\}$ under which the series is the Fourier series of an
integrable function.
Frequently, the series
\begin{eqnarray}\label{cos} a_0/2+\sum\limits_{n=1}^\infty a_n\cos nx \end{eqnarray}
and
\begin{eqnarray}\label{sin} \sum\limits_{n=1}^\infty b_n\sin nx \end{eqnarray}
are investigated separately, since there is a difference in their
behavior. Usually, integrability of (\ref{sin}) requires additional
assumptions. However, one of the basic assumptions is that the
sequence $\{a_n\}$ or $\{b_n\}$ is of bounded variation, written
$\{a_n\}\in bv$ or $\{b_n\}\in bv$, that is, satisfies the condition
(for $\{a_n\}$; similarly for $\{b_n\}$)

\begin{eqnarray}\label{sbv}\sum\limits_{n=1}^\infty|\Delta a_n|<\infty,\end{eqnarray}
where $\Delta a_n=a_n-a_{n+1}$ and similarly for $\Delta b_n.$

One of the strongest known conditions that ensures (along with
certain natural assumptions) the integrability of trigonometric
series, pulled in \cite{AF, BuTa}, can be described as follows. Let
the space of sequences $\{d_n\}$ be endowed with the norm

\begin{eqnarray}\label{amseq1}\|\{d_n\}\|_{a_{1,2}}=\sum\limits_{m=0}^\infty
\left\{\sum\limits_{j=1}^\infty\left[\sum\limits_{n=j2^m}^{(j+1)2^m-1}
|d_n|\right]^2\right\}^{1/2}<\infty.                \end{eqnarray}
It is of amalgam nature; the reader can consult on the theory of
various amalgam spaces in \cite{Fei, FS, Heil}.

It is proved in \cite{AF} and \cite{BuTa} that if the coefficients
$\{a_n\}$ in (\ref{cos}) tend to $0$ as $n\to\infty$ and the
sequence $\{\Delta a_n\}$ is in $a_{1,2},$ then (\ref{cos})
represents an integrable function on $[0,\pi].$ In parallel,
if $\{\Delta b_n\}\in a_{1,2},$ then (\ref{sin}) represents an
integrable function on $[0,\pi]$ if and only if

\begin{eqnarray}\label{condsin}\sum\limits_{n=1}^\infty\frac{|b_n|}{n}<\infty.\end{eqnarray}

Let us introduce the function space $A_{1,2}$ similar to
(\ref{amseq1}). We say that a locally integrable function $g$
defined on $\mathbb R_+$ belongs to $A_{1,2}$ if

\begin{eqnarray}\label{amfun}\|g\|_{A_{1,2}}=\sum\limits_{m=-\infty}^\infty
\left\{\sum\limits_{j=1}^\infty\left[\int\limits_{j2^m}^{(j+1)2^m}|g(t)|
\,dt\right]^2\right\}^{1/2}dx<\infty.              \end{eqnarray}
This space is of amalgam nature as well, since each of the values we
integrate after denotes the norm in the Wiener amalgam space
$W(L^1,\ell^2)$ for functions $2^mg(2^mt),$ where $\ell^p$ is a
space of sequences $\{d_j\}$ endowed with the norm

\begin{eqnarray*} \|\{d_j\}\|_{\ell^2}=\biggl(\sum\limits_{j=1}^\infty|d_j|^2\biggr)^{1/2}\end{eqnarray*}
and the norm of a function $g:\mathbb R_+\to \mathbb C$ from the
amalgam space $W(L^1,\ell^2)$ is taken as

\begin{eqnarray*}\|\{\int\limits_j^{j+1}|g(t)|\,dt\}\|_{\ell^2}.\end{eqnarray*}
 In other words, we can rewrite (\ref{amfun}) as follows:

\begin{eqnarray}\label{amfunn}\|g\|_{A_{1,2}}= \sum\limits_{m=-\infty}^\infty
\,\|2^mg(2^m\cdot)\|_{W(L^1,\ell^2)} < \infty.\end{eqnarray}

The paper is organized as follows. Section 2 deals  with a
preliminary results on embedding of $A_{1,2}$ in $L^1$. In Section 3
we prove our main results on integrability of the Fourier
transforms. Then, in Section 4, we use some of the obtained results
to get an extension of the above mentioned results on integrability
of trigonometric series (\cite{AF, BuTa}).

Of course, the authors are aware of the multidimensional
generalization in \cite{AF1}. A multidimensional version of our
results is in work and will appear elsewhere.

Here and in what follows $\varphi\lesssim \psi$ means that
$\varphi\le C\psi$ with $C$ being an absolute constant.

\bigskip

\section{$A_{1,2}$ is a subspace of $L^1(\mathbb R_+)$}

We start with proving that the considered space is a subspace of
$L^1(\mathbb R_+).$

\begin{lemma}\label{bv} There holds $A_{1,2}\subset L^1(\mathbb R_+).$ \end{lemma}

\begin{proof} We have

\begin{eqnarray*}\ln\frac{3}{2}\int\limits_0^\infty|g(t)|\,dt&=&\int\limits_0^\infty\frac{1}{x}
\int\limits_{2/x}^{3/x}|g(t)|\,dt\,dx\\
&\le&\sum\limits_{m=-\infty}^\infty\int\limits_{2^m}^{2^{m+1}}\frac{1}{x}
\int\limits_{2/x}^{3/x}|g(t)|\,dt\,dx.
\end{eqnarray*}
Since $\frac{2}{x}\ge2^{-m}$ and $\frac{3}{x}\le3(2^{-m}),$ the
right-hand side does not exceed

\begin{eqnarray*}\sum\limits_{m=-\infty}^\infty\int\limits_{2^m}^{2^{m+1}}\frac{1}{x}
\int\limits_{2^{-m}}^{3(2^{-m})}|g(t)|\,dt\,dx,    \end{eqnarray*}
which, in turn, does not exceed

\begin{eqnarray*}3\sum\limits_{m=-\infty}^\infty\int\limits_{2^m}^{2^{m+1}}\frac{1}{x}
\sum\limits_{j=1}^\infty\frac{1}{j}\int\limits_{j2^{-m}}^{(j+1)2^{-m}}|g(t)|\,dt\,dx.
\end{eqnarray*}
Indeed, the integral over $(2^{-m},3(2^{-m}))$ can be split into two
ones over $(2^{-m},2(2^{-m}))$ and $(2(2^{-m}),$ $3(2^{-m}))$, and
the first one is the first summand of the sum in $j$, while the
second one is twice the second summand of that sum.

Applying now the Schwarz-Cauchy-Bunyakovskii inequality yields

\begin{eqnarray*}\int\limits_0^\infty|g(t)|\,dt\lesssim
\sum\limits_{m=-\infty}^\infty\left\{\sum\limits_{j=1}^\infty
\left[\int\limits_{j2^m}^{(j+1)2^m}|g(t)|\,dt\right]^2\right\}^{1/2}dx,
\end{eqnarray*}
the desired bound. \hfill\end{proof}

\begin{remark}\label{rembv} The following different form of Lemma \ref{bv} is
useful in applications: if $f'\in A_{1,2}$, then $f$ is of bounded
variation, that is, $f'\in L^1(\mathbb R_+)$.
\end{remark}

Before proceeding to estimates of Fourier transforms let us compare
$A_{1,2}$ with a space important in such problems. In the study of
integrability properties of the Fourier transform, the following
$T$-transform of a function $g$ defined on $(0,\infty)$ is of
importance

\begin{eqnarray}\label{Ttr} Tg(t)=\int\limits_0^{t/2}
\frac{g(t-s)-g(t+s)}{s}\,ds= \int\limits_{t/2}^{3t/2}
\frac{g(s)}{t-s}\,ds,                            \end{eqnarray}
where the integral is understood in the Cauchy principal value
sense.

 Note that the $T$-transform is related to the Hilbert transform given by
 \begin{eqnarray}\label{Ttr} Hg(t)=\int\limits_0^{\infty}
\frac{g(t-s)-g(t+s)}{s}\,ds=\int\limits_{0}^{\infty}
\frac{g(s)}{t-s}\,ds.
\end{eqnarray}
This is revealed and discussed in \cite{L0}, \cite{Fr}, etc. For
example, it was shown in the mentioned works that for the Hilbert
transform $Hg$ of an integrable odd function, one has for $x>0$

\begin{eqnarray*}Hg(x)=Tg(x)+O(\Gamma(x)),\end{eqnarray*}
where $\int_{\mathbb R_+}|\Gamma(x)|\,dx\le\int_{\mathbb
R_+}|g(x)|\,dx.$ On the other hand, let us remark  that there is an
essential difference between $H$ and $T$-transforms since, e.g., for
the characteristic function $\chi_{[0, d]}$, $d > 0$, we have
(\cite{Fr})

$$\|H \chi[0, \delta]\|_{L^1(\R)}=\infty\qquad\mbox{and}\qquad \|T \chi[0, \delta]
\|_{L^1(\R_+)}= \delta\ln 3.$$

Now, the space $BT$ that has proved to be of importance and is
related to the real Hardy space, is the space of all the functions
that are integrable along with its $T$-transform. Indeed (see
\cite{L0}), a function of bounded variation on $\mathbb R_+$ with
the derivative from $BT$ has the integrable cosine Fourier
transform. As for the sine transform, an analog of (\ref{condsin})
should be added.

Note that the classes  $BT$ and $A_{1,2}$ are incomparable.
Counterexamples for sequences can be found in \cite{AF} and
\cite{Fr0}).

\bigskip

\section{Main results}

We study, for $\gamma=0$ or $1$, the Fourier transform

\begin{eqnarray*} \widehat {f}_\gamma(x)=\int\limits_0^\infty
f(t)\cos(xt-\frac{\pi\gamma}{2})\,dt.              \end{eqnarray*}
It is clear that $\widehat {f}_\gamma$ represents the cosine Fourier
transform in the case $\gamma=0$ while taking $\gamma=1$ gives the
sine Fourier transform.

\begin{theorem}\label{main} Let $f$ be locally absolutely continuous
on $\mathbb R_+$ and vanishing at infinity, that is,
$\lim\limits_{t\to\infty}f(t)=0,$ and $f'\in A_{1,2}$. Then for
$x>0$

\begin{eqnarray}\label{ftmain} \widehat {f}_\gamma(x)=
\frac{1}{x}f(\frac{\pi}{2x})\sin\frac{\pi\gamma}{2}+\theta\Gamma(x),
\end{eqnarray}
where $\gamma=0$ or $1$, and $\|\Gamma\|_{L^1(\mathbb
R_+)}\lesssim\|f'\|_{A_{1,2}}.$ \end{theorem}

\begin{proof}
Splitting the integral and integrating by parts, we obtain

\begin{eqnarray*} \widehat {f}_\gamma(x)=-
\frac{1}{x}f(\frac{\pi}{2x})\sin\frac{\pi}{2}(1-\gamma)+
\int\limits_0^{\frac{\pi}{2x}} f(t)\cos(xt-\frac{\pi\gamma}{2})\,dt
-\frac{1}{x}\int\limits_{\frac{\pi}{2x}}^\infty
f'(t)\sin(xt-\frac{\pi\gamma}{2})\,dt.
\end{eqnarray*}
Further,

\begin{eqnarray*}&\quad&\int\limits_0^{\frac{\pi}{2x}}
f(t)\cos(xt-\frac{\pi\gamma}{2})\,dt\\
&=&\int\limits_0^{\frac{\pi}{2x}}[ f(t)-f(\frac{\pi}{2x})]
\cos(xt-\frac{\pi\gamma}{2})\,dt+\int\limits_0^{\frac{\pi}{2x}}
f(\frac{\pi}{2x})\cos(xt-\frac{\pi\gamma}{2})\,dt\\
&=&-\int\limits_0^{\frac{\pi}{2x}}\left[\int\limits_t^{\frac{\pi}{2x}}
f'(s)\,ds\right]\cos(xt-\frac{\pi\gamma}{2})\,dt+\frac{1}{x}f(\frac{\pi}{2x})
\sin\frac{\pi}{2}(1-\gamma)+\frac{1}{x}f(\frac{\pi}{2x})\sin\frac{\pi\gamma}{2}\\
&=&\frac{1}{x}f(\frac{\pi}{2x})\sin\frac{\pi\gamma}{2}+\frac{1}{x}f(\frac{\pi}{2x})
\sin\frac{\pi}{2}(1-\gamma)+O\left(\int\limits_0^{\frac{\pi}{2x}}s|f'(s)|\,ds\right).
\end{eqnarray*}
Since

\begin{eqnarray}\label{near0}\int\limits_0^\infty\int\limits_0^{\frac{\pi}{2x}}s|f'(s)|\,ds\,dx
=\frac{\pi}{2}\int\limits_0^\infty|f'(s)|\,ds,       \end{eqnarray}
it follows from Lemma \ref{bv} that to prove the theorem it remains
to estimate

\begin{eqnarray}\label{hardint}\int\limits_0^\infty\frac{1}{x}
\left|\int\limits_{\frac{\pi}{2x}}^\infty
f'(t)\sin(xt-\frac{\pi\gamma}{2})\,dt\right|\,dx.   \end{eqnarray}
We can study

\begin{eqnarray}\label{hardint1}\sum\limits_{m=-\infty}^\infty\int\limits_{2^m}^{2^{m+1}}\frac{1}{x}
\left|\,\int\limits_{2^{-m}}^\infty
f'(t)\sin(xt-\frac{\pi\gamma}{2}) \,dt\right|\,dx   \end{eqnarray}
instead. Indeed,

\begin{eqnarray}\label{hardint2}\sum\limits_{m=-\infty}^\infty\int\limits_{2^m}^{2^{m+1}}\frac{1}{x}
\left|\,\int\limits_{2^{-m}}^{\frac{\pi}{2x}} f'(t)
\sin(xt-\frac{\pi\gamma}{2}) \,dt\right|\,dx&\le&
\sum\limits_{m=-\infty}^\infty\int\limits_{2^m}^{2^{m+1}}\frac{1}{x}
\int\limits_{\frac{1}{x}}^{\frac{\pi}{2x}}|f'(t)|\,dt\,dx\nonumber\\
&\le&\int\limits_0^\infty|f'(t)|\,dt.             \end{eqnarray}
Here and in other estimates with (\ref{hardint1}) there is no
difference between the sine and cosine, therefore all will follow
from the next result the statement and the proof of which are
inspired by Lemma 2 in \cite{AF}.

\begin{lemma}\label{bale} Let $g$ be an integrable function on
$\mathbb R_+.$ Then for $m=0,\pm1,\pm2,...$

\begin{eqnarray}\label{hardint3}\int\limits_{2^m}^{2^{m+1}}\frac{1}{x}
\left|\,\int\limits_{2^{-m}}^\infty g(t)e^{-ixt} \,dt\right|\,dx
\lesssim\left(\,\sum\limits_{j=1}^\infty\left[\int\limits_{j2^{-m}}^{(j+1)2^{-m}}
|g(t)|\,dt\right]^2\right)^{1/2}.
\end{eqnarray}\end{lemma}

\begin{proof}[Proof of Lemma \ref{bale}] We have

\begin{eqnarray*}\int\limits_{2^m}^{2^{m+1}}\frac{1}{x}
\left|\,\int\limits_{2^{-m}}^\infty g(t)e^{-ixt} \,dt\right|\,dx
\lesssim\int\limits_{2^m}^{2^{m+1}}\left|\frac{\sin2^{-m}x}{x}
\int\limits_{2^{-m}}^\infty g(t)e^{-ixt} \,dt\right|\,dx.
\end{eqnarray*}

By the Schwarz-Cauchy-Bunyakovskii inequality the right-hand side
does not exceed

\begin{eqnarray}\label{scb}&\quad&\left(\int\limits_{2^m}^{2^{m+1}}\,dx\right)^{1/2}
\left(\int\limits_{2^m}^{2^{m+1}}\,\left|\frac{\sin2^{-m}x}{x}
\int\limits_{2^{-m}}^\infty g(t)e^{-ixt} \,dt\right|^2\,dx\right)^{1/2}\nonumber\\
&=&2^{m/2}\left(\int\limits_{2^m}^{2^{m+1}}\,\left|\frac{\sin2^{-m}x}{x}
\int\limits_{2^{-m}}^\infty g(t)e^{-ixt}\,dt\right|^2\,dx
\right)^{1/2}.\end{eqnarray}             Denoting

\begin{eqnarray*} S_m(x)=\frac{\sin2^{-m}x}{x};\qquad
G_m(x)=\int\limits_{2^{-m}}^\infty g(t)e^{-ixt} \,dt,\end{eqnarray*}
we have to estimate

\begin{eqnarray*}2^{m/2}\left(\int\limits_{2^m}^{2^{m+1}}\,\left|S_m(x)G_m(x)
\right|^2\,dx \right)^{1/2}.                      \end{eqnarray*}
The product $S_m(x)G_m(x)$ is the convolution of the inverse Fourier
transforms, and we get

\begin{eqnarray}\label{conv}2^{m/2}\left(\int\limits_{2^m}^{2^{m+1}}\,
\left|\check{S_m}\ast\check{G_m}(x) \right|^2\,dx \right)^{1/2}.
\end{eqnarray}
Let us take into account that

\begin{eqnarray}\label{f2}\int_0^\infty\frac{\sin ax}{x}\,\cos yx
\,dx=\begin{cases}\frac\pi2, & y<a\\ \frac\pi4, & y=a\\
0, & y>a\end{cases};\end{eqnarray}

\begin{remark}\label{dir} This formula goes back to Fourier, see, e.g.,
Remark 12 in the cited literature of \cite{bochner}.\end{remark}

Since all our estimates will be for integrals, the value at one
point is of no importance. Therefore, using (\ref{f2}) for

\begin{eqnarray*}\check{G_m}(x)=\frac2\pi\int\limits_0^\infty
\frac{\sin2^{-m}u}{u}\cos xu\,du,                  \end{eqnarray*}
we may consider $\check{S_m}(x)$ to be $1$ for $x<2^{-m}$ and zero
otherwise.

Further,

\begin{eqnarray*}G_m(x)=\sum\limits_{j=1}^\infty\int\limits_{j2^{-m}}^{(j+1)2^{-m}}
g(t)e^{-ixt}\,dt=\sum\limits_{j=1}^\infty G_{m,j}(x),\end{eqnarray*}
where

\begin{eqnarray*}G_{m,j}(x)=\int\limits_{j2^{-m}}^{(j+1)2^{-m}}
g(t)e^{-ixt}\,dt.\end{eqnarray*}    Correspondingly,

\begin{eqnarray*}\check{G_m}(x)=\sum\limits_{j=1}^\infty g_{m,j}(x),\end{eqnarray*}
with $g_{m,j}(x)=g(x)$ when $j2^{-m}\le x\le(j+1)2^{-m}$ and zero
otherwise.

Recall now that Young's inequality for convolution reads as follows
(see, e.g., \cite[Ch.V, \S 1]{SW}): If $\varphi\in L^r(\mathbb R)$
and $\psi\in L^q(\mathbb R)$, then for $\frac{1}{r}+\frac{1}{q}=
\frac{1}{p}+1,$ $1\le p,q,r\le\infty,$

\begin{eqnarray}\label{young}\|\varphi\ast\psi\|_p\le\|\varphi\|_r\|\psi\|_q.\end{eqnarray}

To apply these, we may regard the supports $I_j$ of
$\check{S_m}(x)\ast\check{G_{m,j}}$ as non-overlapping. Otherwise,
we can consider separately sums over $j$-s odd only and over $j$-s
even, as in the proof of Lemma 2 in \cite{AF}, arrive at the same
upper bounds and then combine them.

Thus,

\begin{eqnarray*}2^{m/2}\left(\int\limits_{2^m}^{2^{m+1}}\,
\left|\check{S_m}\ast\check{G_m}(x) \right|^2\,dx \right)^{1/2}&\le&
2^{m/2}\left(\int\limits_{\mathbb
R}\,\left|\check{S_m}\ast\check{G_m}(x) \right|^2\,dx\right)^{1/2}\\
\le2^{m/2}2\left(\sum\limits_{j=1}^\infty\int\limits_{I_j}\,
\left|\check{S_m}\ast\check{G_m}(x) \right|^2\,dx \right)^{1/2}
&\lesssim& 2^{m/2}\left(\sum\limits_{j=1}^\infty\int\limits_{\mathbb
R}\,\left|\check{S_m}\ast\check{G_m}(x) \right|^2\,dx \right)^{1/2}.
\end{eqnarray*}
Applying Young's inequality with $\varphi=\check{S_m}$ and
$\psi=g_{m,j},$ $q=1$ and $p=r=2$, we obtain

\begin{eqnarray*}2^{m/2}\left(\sum\limits_{j=1}^\infty\int\limits_{\mathbb
R}\,\left|\check{S_m}\ast\check{G_m}(x) \right|^2\,dx \right)^{1/2}
\lesssim2^{m/2}\left(\sum\limits_{j=1}^\infty\,\|\check{S_m}\|_2^2\,
\|g_{m,j}\|_1^2\right)^{1/2}.\end{eqnarray*}        Since

\begin{eqnarray*}\|\check{S_m}\|_2=\left(\int\limits_0^{2^{-m}}
dx\right)^{1/2}=2^{-m/2},\end{eqnarray*}  we get the required bound

\begin{eqnarray*}\left(\,\sum\limits_{j=1}^\infty\left[\int\limits_{j2^{-m}}^{(j+1)2^{-m}}
|g(t)|\,dt\right]^2\right)^{1/2}.\end{eqnarray*} \hfill\end{proof}

Applying now the proven lemma, we complete the proof of the theorem,
since $m$ runs from $-\infty$ to $\infty$ and we can write $2^m$
instead of $2^{-m}$. \hfill\end{proof}

\bigskip

\section{Applications}

As an application, we obtain integrability results for (\ref{cos})
and (\ref{sin}) given above, for (\ref{sin}) even in a stronger,
asymptotic form.

We can relate this problem to a similar one for Fourier transforms
as follows. First, given series (\ref{cos}) or (\ref{sin}) with the
null sequence of coefficients being in an appropriate sequence
space, set for $x \in [n,n+1]$

\begin{eqnarray*} A(x)&=&a_n + (n-x) \Delta a_n, \qquad a_0 =0, \\
B(x)&=&b_n + (n-x) \Delta b_n.                    \end{eqnarray*}
So, we construct a corresponding function by means of linear
interpolation of the sequence of coefficients. Of course, one may
interpolate not only linearly, but there are no problems where this
might be of importance so far.

Secondly, for functions of bounded variation $\varphi$, to pass from
series to integrals and vice versa, we will make use of the
following result due to Trigub \cite[Th. 4]{T2} (see also \cite{TB};
an earlier version, for functions with compact support, is due to
Belinsky \cite{Be}):

\begin{eqnarray}\label{insevar} \sup\limits_{0<|x| \le \pi} \biggl| \int_{-\infty}^{+\infty}
\varphi (t) e^{- ixt} \, dt - \sum\limits_{-\infty}^{+\infty}
\varphi (k) e^{- ikx} \biggl| \;\;\lesssim\;  \|\varphi\|_{BV}.
\end{eqnarray}
The relation (\ref{insevar}) allows us to pass from estimating
trigonometric series (\ref{cos}) and (\ref{sin}) to estimating the
Fourier transform of $A(t)$ and $B(t),$ respectively.

\begin{theorem}\label{appl} If
the coefficients $\{a_n\}$ in (\ref{cos}) and $\{b_n\}$ in
(\ref{sin}) tend to $0$ as $n\to\infty,$ and the sequences $\{\Delta
a_n\}$ and $\{\Delta b_n\}$ are in $a_{1,2},$ then (\ref{cos})
represents an integrable function on $[0,\pi],$ and

\begin{eqnarray}\label{assin}\sum\limits_{n=1}^\infty b_n\sin nx
=\frac{1}{x}B(\frac{\pi}{2x})+\Gamma(x),          \end{eqnarray}
where $\int\limits_0^\pi |\Gamma(x)|\,dx\lesssim \|\{\Delta
b_n\}\|_{a_{1,2}}.$
\end{theorem}

\begin{proof} First of all, we observe that, as usual, no finite number of
coefficients can affect the behavior of trigonometric series.
Therefore, we have to estimate

\begin{eqnarray}\label{amfun1}\sum\limits_{m=1}^\infty
\left\{\sum\limits_{j=1}^\infty\left[\sum\limits_{n=j2^{-m}}^{(j+1)2^{-m}-1}
\int\limits_{n}^{n+1}|g(t)|\right]^2\right\}^{1/2}<\infty.
\end{eqnarray}
By (\ref{insevar}) we can deal with the functions $A(x)$ and $B(x)$
instead of the given sequences. Taking in (\ref{amfun1})
$g(t)=A'(t)$ that equals $-\Delta a_n$ when $n\le t<n+1,$ we prove
the first part of the theorem by fulfilling routine calculations.

We treat the remainder term in (\ref{assin}) in the same manner.
Finally, to make sure that (\ref{condsin}) follow from the obtained
asymptotic representation, we get

\begin{eqnarray*} \int\limits_1^N\frac{1}{x}|B(\frac{\pi}{2x})|\,dx
=\sum\limits_{n=1}^{N-1}\int\limits_n^{n+1}\frac{1}{x}|B(\frac{\pi}{2x})|\,dx
=\sum\limits_{n=1}^{N-1}\int\limits_{\frac{\pi}{2(n+1)}}^{\frac{\pi}{2n}}
\frac{1}{t}|B(t)|\,dt.                             \end{eqnarray*}
The right-hand side is obviously equivalent to
$\sum\limits_{n=1}^N\frac{|b_n|}{n}.$ Letting $N\to \infty$ leads to
(\ref{condsin}) and completes the proof.  \hfill\end{proof}

\bigskip

\section{Acknowledgements}

This paper was started during the visit of the author to the Centre
de Recerca Matem\`{a}tica, Bellaterra, Barcelona, Spain in February
2010. He appreciated the creative atmosphere of CRM and thanks ESF
Network Programme HCAA for supporting this visit.

\end{document}